\documentclass[10pt, a4paper]{article}
\usepackage[utf8]{inputenc}
\usepackage{amsfonts}
\usepackage{amsmath}
\usepackage{amsthm}
\usepackage{amssymb}
\usepackage[font=small,labelfont=bf]{caption}
\usepackage{comment}
\usepackage{xcolor}
\usepackage[colorlinks]{hyperref}        
\hypersetup{citecolor=blue}
\usepackage{mathrsfs}
\usepackage[margin=1in]{geometry}
\usepackage{graphicx}
\usepackage{wrapfig}
\usepackage{cleveref}
\usepackage[round]{natbib}
\usepackage{enumerate}
\usepackage{float}

\newcommand{\agw}{\alpha^\star_{\GW}}
\newcommand{\argw}{\alpha^\star_{\RGW}}

\newcommand{\argwprime}{\tilde{\alpha}^{\star}_{\RGW}}
\newcommand{\argwhat}{\hat{\alpha}^{\star}_{\RGW}}

\newcommand{\Ber}{\mathrm{Ber}}

\newcommand{\dd}{\mathrm{d}}
\newcommand{\E}{\mathbb E}

\newcommand{\calE}{\mathcal E}
\newcommand{\ee}{\varepsilon}
\newcommand{\Eg}{E^\star}
\newcommand{\Egw}{E^\GW_\calQ}

\newcommand{\Egwhat}{E^\GW_{\calQhat}}
\newcommand{\Egwhatdelta}{E^\GW_{\calQhat_\delta}}
\newcommand{\Ergw}{E^\RGW_\calQ}

\newcommand{\Ergwhat}{E^\RGW_{\calQhat}}
\newcommand{\Ergwhatdelta}{E^\RGW_{\calQhat_\delta}}

\newcommand{\KL}{\mathrm{KL}}
\newcommand{\kl}{\mathrm{kl}}

\newcommand{\calP}{\mathcal P}
\newcommand{\calPprime}{\tilde{\mathcal P}}

\newcommand{\calQ}{\mathcal Q}
\newcommand{\calQprime}{\tilde{\mathcal{Q}}}
\newcommand{\calQhat}{\hat{\mathcal{Q}}}

\newcommand{\R}{\mathbb R}

\newcommand{\scrF}{\mathscr{F}}
\newcommand{\scrP}{\mathscr{P}}

\newcommand{\calX}{\mathcal X}

\newcommand{\alphamin}{\alpha_{\min}}
\newcommand{\alphamax}{\alpha_{\max}}

\newcommand{\fab}{F_{\alpha,\beta}}
\newcommand{\fabc}{F^c_{\alpha,\beta}}
\newcommand{\gab}{G_{\alpha,\beta}}

\newcommand{\diff}{\,\mathrm{d}} 

\newcommand{\GRO}{\mathrm{GRO}} 
\newcommand{\GROW}{\mathrm{GROW}} 
\newcommand{\REGROW}{\mathrm{REGROW}} 

\newcommand{\RGW}{\mathrm{RGW}} 
\newcommand{\GW}{\mathrm{GW}} 
\renewcommand{\phi}{\varphi}

\makeatletter
\def\secondtitle#1#2#3{
	\clearpage
	\setcounter{page}{1}
	\gdef\@title{#1}
	\gdef\@author{#2}
	\gdef\@date{#3}
	\maketitle
}
\makeatother

\newtheorem{example}{Example}
\newtheorem{lemma}{Lemma}
\newtheorem*{lemma*}{Lemma}
\newtheorem{proposition}{Proposition}
\newtheorem{theorem}{Theorem}
\newtheorem{corollary}{Corollary}

\theoremstyle{remark} 
\newtheorem{remark}{Remark}

\begin{document}
	\makeatletter
	\renewcommand*{\@fnsymbol}[1]{\ensuremath{\ifcase#1\or*\or\mathsection\or\else\@ctrerr\fi}}
	\makeatother
	\title{
		Optimal e-values for testing the mean of a bounded random variable against a composite alternative}
	\author{Sebastian Arnold\thanks{Machine Learning Group, Centrum Wiskunde \& Informatica (CWI), \texttt{sebastian.arnold@cwi.nl}} \and Eugenio Clerico\thanks{Department of Statistics, University of Oxford, \texttt{eugenio.clerico@gmail.com}}}
	\maketitle
	\begin{abstract}
		We derive explicitly the  e-values with optimal (relative) growth rate \emph{in the worst case} for testing the mean of a bounded random variable, thereby providing the first application of the (RE)GROW quality criteria beyond the assumption of mutually absolutely continuous hypotheses for e-values originally proposed by Gr\"unwald et al.\ (2024). For both criteria, we explicitly characterise the alternatives that are most difficult to test against and show that they admit a meaningful interpretation. We give two important examples in which REGROW provides a powerful quality criterion to choose optimal e-variables whereas GROW leads to trivial solutions.
	\end{abstract}
	
	Keywords: \emph{e-variable; optimal growth rate; nonparametric mean-testing; composite alternative.}

	\section{Introduction}
	E-values are a statistical tool for measuring evidence against a null hypothesis that has attracted increasing interest in recent years as a flexible and safe alternative to p-values; see, e.g., \cite{grunwald2024safe}, \cite{vovk2021evalues}, and \cite{ramdas2023gametheoretic}. An \emph{e-value} is the observed value (on the data) of an \emph{e-variable}, a non-negative random variable with expectation  bounded by $1$ under the null hypothesis $\calP$. Markov’s inequality ensures that large e-values are unlikely under the null, and the magnitude of an observed e-value directly quantifies statistical evidence against the null.
	We refer to  \cite{ramdas2025hypothesis} for a comprehensive introduction to the subject.

	A natural question when testing with e-values is how to choose the right e-variable to use. This is typically addressed by specifying an alternative hypothesis $\calQ$ representing the distributions under which we would particularly like the test to reject the null. 
	When  $\mathcal Q$ is \emph{simple} (namely, it consists of a single distribution $Q$), the most commonly adopted notion of `power' for an e-variable $E$ is the \emph{expected growth rate} (or \emph{e-power}) $\mathbb E_Q[\log E]$, which quantifies the rate at which evidence accumulates under the alternative. Maximising expected logarithmic growth is well motivated by gambling theory, information theory, and finance, and goes back to the seminal work of \citet{kelly1956a}; see also \citet{cover2006elements} and \cite{shafer2021testing}.
	A central result in the theory of e-values states that there exists a ($Q$-almost surely) unique e-variable $\Eg_Q$ maximising $\mathbb E_Q[\log E]$ among all valid e-variables. This optimal $\Eg_Q$, known as the \emph{growth-rate optimal} (GRO) or \emph{numéraire} e-variable under $Q$, can be expressed as a generalised likelihood ratio between $Q$ and a (sub)probability measure obtained via a generalised \emph{reverse information projection} of $Q$ onto a convexification of the null~\citep{grunwald2024safe,larsson2025the}.
	
	Whereas expected growth rate is a natural optimality criterion for a simple alternative, extending it to a \emph{composite} alternative (i.e., a family of distributions) is less straightforward. Following \cite{grunwald2024safe}, we consider two criteria: \emph{GROW} (growth-rate optimal in the worst case) and \emph{REGROW} (relative growth-rate optimal in the worst case). Under a finite-KL condition, \citet{grunwald2024safe} establish  existence and almost-sure uniqueness of the corresponding optimal e-variables when all measures in the null and alternative admit densities with respect to a common reference measure, a condition we call the \emph{absolute continuity (AC)} assumption. General existence and characterisation results beyond this dominated setting remain largely unexplored.
	
	\textbf{Main contribution.} In this paper, we study the GROW and REGROW criteria for testing the mean of a bounded real-valued random variable. More precisely, we consider hypotheses of the form $\mathcal P = \{\text{The mean is ($\leq$) $\mu_0$}\}$ against $\mathcal Q = \{\text{The mean is ($>$) $\mu_1$}\}$, and variations of these, where both $\mathcal P$ and $\mathcal Q$ consist of distributions supported on a common bounded interval. These hypotheses do not satisfy the AC assumption. Nevertheless, we show that the corresponding GROW and REGROW e-variables are well defined and derive explicit characterisations of the optimal e-variables and their associated worst-case distributions.
	
	Our analysis combines an existing admissibility reduction with a new convex-analytic reformulation. First, we leverage the fact that the \emph{coin-betting e-class}, a family of e-variables parametrised by a single scalar, forms the minimal admissible class for our testing problem \citep{clerico2025on,wang2026ebacktesting}. This reduces the optimisation over all e-variables to a one-dimensional problem. Our key step is then to make the convex structure of the worst-case optimisation explicit by rewriting it in terms of convex envelopes. Analysing these envelopes identifies the relevant e-variables and worst-case distributions, reducing the final optimisation to balancing two monotone terms. Beyond yielding explicit solutions in the present setting, our proof ideas may provide a useful strategy for studying (RE)GROW for a broader class of hypothesis testing problems.
	
	\textbf{Notation and definitions.} 
	Let $\calX=[A,B]$, for $-\infty <A<B<\infty$, endowed with the standard relative topology and its Borel sigma-field. We let $\scrP$ be the set of Borel probability measures on $\calX$, and $\scrF$ the family of Borel measurable functions $f:\calX\to\R$. For $P\in \scrP$ and $f\in\scrF$ integrable with respect to $P$, we denote by $\E_P[f(X)]$ (or $\E_P[f]$) the expectation $\int_{A}^B f(x) \diff P(x)$ of $f$ under $P$. Given a \emph{null hypothesis} $\calP \subseteq \scrP$, $E\in\scrF$ is an \emph{e-variable} (for $\calP$) if it is non-negative, and $\E_P[E]\leq1$ for all $P\in \calP$. We let $\calE(\calP)\subseteq \scrF$ be the family of all e-variables for $\calP$. We call $P\in \scrP$ \emph{absolutely continuous} with respect to a Borel measure $\nu$ on $\calX$, if it admits a density with respect to $\nu$, in which case we write $P\ll \nu$. For $x\in [A,B]$, $\delta_x$ is the Dirac measure at $x$.

	\section{Preliminaries}\label{sec:(RE)GRO(W)}
	Let $\calP\subseteq\scrP$ be a fixed null hypothesis.
	For a simple alternative $\calQ = \{Q\}$, with $Q\in \scrP$, we define the \emph{e-power} of an e-variable $E\in \calE(\calP)$ as $\E_Q[\log E]$, with the convention $\E_Q[\log E]=-\infty$ whenever $\E_Q[\max\{0,-\log E\}] = \infty$. Under the AC assumption and a finite-KL condition, \cite{grunwald2024safe} showed that the e-power admits a unique (up to $Q$-null sets) maximiser over $\calE(\calP)$. We call this maximiser the \emph{GRO e-variable} for $Q$ and denote it by $\Eg_Q$.
	In a follow-up work, \cite{larsson2025evariables} extended this result, showing that $\Eg_Q$ exists and is well-defined under essentially no restriction on $\calP$ and $Q$. We denote the e-power of $\Eg_Q$ by $\GRO(Q)$, 
	\begin{equation*}
		\GRO(Q) = \sup_{E \in \mathcal{E}(\calP)} \E_{Q} [\log E] = \E_Q[\log \Eg_Q]\,.
	\end{equation*}

	If the alternative $\calQ$ is composite, that is, it consists of several distributions, the problem of how to pick the best e-variable becomes more complex. 
	Under the GROW criterion, we consider the worst-case  $\inf_{Q\in\calQ}\E_Q[\log E]$ for each $E$, and define the \emph{optimal growth rate in the worst case} 
	\begin{equation}\label{eq:def_GROW}
		\GROW(\calQ) = \sup_{E \in \calE(\calP)} \inf_{Q\in \calQ}\E_{Q} [\log E]\,. 
	\end{equation} If a maximiser of~\eqref{eq:def_GROW} exists (under the AC assumption, the first generalisation of Theorem 1 in \citealt{grunwald2024safe} ensures existence and uniqueness), we call it the GROW e-variable under $\calQ$ and denote it by $\Egw$. 
	While GROW provides a clear and natural optimality principle, it is often too strict and over-pessimistic. Prioritising the least favourable alternative in $\calQ$ comes at the expense of performance under the other alternative measures in $\calQ$, and can push towards a trivial solution that can never reject the null. This happens in particular if the null and the alternative are not properly separated, as illustrated in the example below.
	To address this issue, \cite{grunwald2024safe} introduced the \emph{relative optimal growth-rate in the worst case} as 
	\begin{equation}\label{eq:def_REGROW}
		\REGROW(\calQ) = \sup_{E \in \mathcal{E}(\calP)} \inf_{Q\in \calQ}\E_{Q} [\log E-\GRO(Q)]\,.
	\end{equation}
	When a maximiser for $\REGROW(\calQ)$ exists, we call it the REGROW e-variable, and denote it by $\Ergw$. Under the AC assumption, \cite{grunwald2024safe} established existence and (almost sure) uniqueness of the REGROW e-variable.
	Rather than requiring performance guarantees that hold uniformly over $\calQ$, REGROW evaluates an e-variable relative to the best performance achievable under each individual alternative. This  is close in spirit to the idea of regret in online learning, where the quality of a decision is assessed relative to the best decision in hindsight. 
	Evaluating performance relative to the oracle e-variable tailored to the true distribution, REGROW avoids sacrificing easy measures in $\calQ$ merely because some others are harder to distinguish from $\calP$.
	
	The following example illustrates a case where $\Egw$ is powerless, yet $\Ergw$ is non-trivial.

	\begin{example}[Testing a Bernoulli]\label{ex:bern} For
		fixed $\mu_0\in (0,1)$ and $\calP = \{\Ber(\mu_0)\}$, we consider the one-sided alternative $\calQ_1 = \{\Ber(\mu)\mid\mu>\mu_0\}$ and the two-sided alternative $\calQ_2 = \{\Ber(\mu)\mid\mu\neq\mu_0\}$. We can identify each e-variable with a vector $Z\in\R_{\geq0}^2$ via the embedding $E\mapsto (E(1), E(0))^\top$. It is easily checked that $\calE(\calP)$ consists of all non-negative vectors component-wise dominated by some $Z_\ee = \big({\ee}/{\mu_0},{(1-\ee)}/{(1-\mu_0)}\big)^\top$, $\ee\in[0,1]$, and that $Z_\mu$ is the GRO e-variable for the simple alternative $\{Q_\mu\}=\{\Ber(\mu)\}$, for any $\mu\in[0,1]$. Moreover,
		$$\GRO(Q_\mu) = \E_{Q_\mu}[\log Z_\mu] = \mu\log\frac{\mu}{\mu_0} + (1-\mu)\log\frac{1-\mu}{1-\mu_0} = \kl(\mu, \mu_0)\,,$$ where $\kl(\theta,\theta')=\KL(\Ber(\theta)\|\Ber(\theta'))$, for $\theta,\theta'\in[0,1]$.
		It follows that the trivial growth rates $\GROW(\mathcal Q_1) = \inf_{\mu>\mu_0} \GRO(Q_\mu)=0=\GROW(\mathcal Q_2)$ are achieved by the constant e-variable $E_{\calQ_1}^{\GW}=E_{\calQ_2}^{\GW} = Z_{\mu_0} = (1,1)^\top$. Next, we show that the REGROW criterion provides us with non-trivial e-variables. Fix $\ee\in[0,1]$ and $\mu\neq\mu_0$. Then,
		$$\E_{Q_\mu}[\log Z_\ee - \log Z_\mu] = \mu\log\frac{\ee}{\mu} + (1-\mu)\log\frac{1-\ee}{1-\mu} = -\kl(\mu,\ee)\,.$$
		By convexity of the KL divergence in both its arguments, it follows that
		\begin{align}
			\REGROW(\calQ_1) &= \sup_{\ee\in[0,1]}\inf_{\mu>\mu_0}\big(-\kl(\mu,\ee)\big) = \sup_{\ee\in[0,1]}\min\big\{-\kl(\mu_0,\ee), -\kl(1,\ee)\big\}\,, \label{eq:REGROW_Bernoulli_example}\\
			\REGROW(\calQ_2) &= \sup_{\ee\in[0,1]}\inf_{\mu\neq\mu_0}\big(-\kl(\mu,\ee)\big) = \sup_{\ee\in[0,1]}\min\big\{-\kl(0,\ee), -\kl(1,\ee)\big\}\,.\label{eq:REGROW_Bernoulli_example_2}
		\end{align}
		By monotonicity and concavity, these suprema are uniquely attained at $\ee^\star_1$ and $\ee^\star_2$, respectively, for which the two terms in the minimum coincide. By symmetry, $\ee^\star_2=1/2$, while a direct calculation yields $\ee^\star_1=(1 + (1-\mu_0)\,\mu_0^{\mu_0/(1-\mu_0)})^{-1}$.
		That is, the REGROW e-variables are $E_{\calQ_1}^\RGW = Z_{\ee^\star_1}$ and $E_{\calQ_2}^\RGW = Z_{1/2}$, 
		where $E_{\calQ_1}^\RGW$ is not the trivial $(1,1)^\top$ since $\ee^\star_1>\mu_0$, and where $E_{\calQ_2}^\RGW$ is non-trivial as long as we are not in the symmetric case $\mu_0= 1/2$. 
		Moreover, we see that the `worst-case' distributions, namely the alternatives for which the infima in \eqref{eq:REGROW_Bernoulli_example} and \eqref{eq:REGROW_Bernoulli_example_2} are attained, are given by $\Ber(\mu_0)$ and $\Ber(1)$, and by $\Ber(0)$ and $\Ber(1)$, respectively.  Interestingly, for the two-sided $\calQ_2$, the two worst-case distributions do not depend on $\mu_0$, while, for the one-sided $\calQ_1$, the null appears as a second worst-case distribution alongside the boundary distribution Ber(1). This aligns with our general findings as discussed in the next section. 
		We refer the reader to Figure~\ref{fig:example} in the Supplement for a visual intuition behind the Bernoulli example.
	\end{example}
	
	We now introduce the general framework underlying our main result, and take as sample space an arbitrary bounded interval $\calX=[A,B]$, with $-\infty<A<B<\infty$.
	For fixed $\mu_0\in(A,B)$, we consider the null  $\calP$ consisting of all probability measures on $[A,B]$ with mean $\mu_0$, that is $$\calP = \{P\in\scrP \mid \E_P[X] = \mu_0\}\,.$$ 
	
	Although $\calP$ does not satisfy the AC assumption, $\calE(\calP)$ admits a simple characterisation \citep{clerico2025on}: it consists of all non-negative functions dominated by a 
	\emph{coin-betting e-variable} 
	\begin{equation}\label{eq:def_coin_betting_e-value}
		E_\alpha\,:\,x\mapsto 1+\alpha (x-\mu_0)
	\end{equation} 
	for $\alpha \in I_{\mu_0}=[\alphamin,\alphamax]$, with
	\begin{equation}\label{eq:alphaminmax}\alphamin=(\mu_0-B)^{-1}<0\,, \quad\alphamax=(\mu_0-A)^{-1}>0\,.\end{equation}
	As a consequence, we can restrict our attention to e-variables of the form~\eqref{eq:def_coin_betting_e-value}, and  finding (RE)GROW e-variables for a given alternative reduces to finding the corresponding optimal betting parameters $\alpha\in I_{\mu_0}$. 
	We refer the reader to \cite{orabona2024tight} for an explanation of why functions of the form \eqref{eq:def_coin_betting_e-value} can be interpreted as the outcomes of a betting game on a coin, and to Proposition~\ref{prop:coin_betting_is_an_optimal_class} and Corollary~\ref{cor:onesided} in the Supplementary Material for a detailed exposition on the admissibility of coin-betting e-variables that we leverage in our analysis.
	
	\section{(RE)GROW optimal e-variables for bounded  mean testing} \label{Sec:Results}

	To state our main result, we introduce the following functions on $[A,B]$, with $\alpha,\beta\in I_{\mu_0}$:
	\begin{align*}
		\textstyle
		f_\alpha(x) = \log\big(1\!+\!\alpha(x\!-\!\mu_0)\big)\,, \; \fab(x) = {f_\alpha(x)}-{f_\beta(x)}\,,\;
		\gab(x)  = \frac{(B-x)F_{\alpha,\beta}(A) + (x-A)F_{\alpha,\beta}(B)}{B-A}\,,
	\end{align*}
	with the conventions $\log0=-\infty$ and $0/0=1$. The function $f_\alpha$ will be used to analyse GROW, $\fab$ will naturally arise when characterising REGROW, while $\gab$ is simply the linear interpolation between $\fab(A)$ and $\fab(B)$. 
	
	\begin{theorem}\label{thm:master_thm}
		Let  $\mu_0,\mu_1\in (A,B)$ with $\mu_1>\mu_0$. In the following three cases, the (RE)GROW e-variables $\Egw$ and $\Ergw$ exist.
		\begin{enumerate}[(a)]
			\item For $\calP = \{P\in\scrP\mid\E_P[X]=\mu_0\}$ vs.\ $\calQ = \{Q\in\scrP\mid\E_Q[X]=\mu_1\}$, we have $\Egw= E_{\agw}$ and $\Ergw= E_{\argw}$, with 
			\begin{equation}\label{eq:GROW_betting_parameter}\agw = \frac{\mu_1-\mu_0}{(\mu_0-A)(B-\mu_0)}\,\in (0,\alphamax)\,,\end{equation}
			and $\argw$ the unique solution of $F_{\argw,\alphamax}(\mu_1)=G_{\argw,\agw}(\mu_1)$ with $\argw>\agw$.
			\item  For $\calPprime = \{P\in\scrP\mid\E_P[X]\leq\mu_0\}$ vs.\ $\tilde{\calQ} = \{Q\in\scrP\mid\E_Q[X]>\mu_1\}$, we have $E_{\calQprime}^{\GW}= E_{\agw}$ and $E_{\calQprime}^{\textrm{RGW}}= E_{\argwprime}$, with $\agw$ given in \eqref{eq:GROW_betting_parameter} and
			\begin{equation}\label{eq:alphaexplicit}\argwprime = \frac{(B-\mu_0) - \frac{(\mu_0-A)(B-\mu_1)}{B-A} r}{(\mu_0-A)(B-\mu_0)\left(1+\frac{B-\mu_1}{B-A}r\right)}\,\in (0,\alphamax)\,, \quad\textrm{for } r = \left(\frac{\mu_1-A}{B-A}\right)^{\frac{\mu_1-A}{B-\mu_1}}\,.\end{equation}
			
			\item For $\calP = \{P\in\scrP\mid\E_P[X]=\mu_0\}$ vs.\ $\calQhat = \{Q\in\scrP\mid\E_Q[X]\neq\mu_0\}$, the GROW e-variable $\Egwhat \equiv1$ is trivial, while $\Ergwhat= E_{\argwhat}$, for 
			\begin{equation}\label{eq:def_rgw_parameter_for_two_sided}
				\argwhat = \frac{(A+B)/2-\mu_0}{(\mu_0-A)(B-\mu_0)}\in I_{\mu_0}\,.
			\end{equation}
			Moreover, for any $\delta>0$ such that $(\mu_0-\delta,\mu_0+\delta) \subseteq (A,B)$, for $\calP$ vs.\ $\calQhat_\delta = \{Q\in\scrP\mid |\E_Q[X]-\mu_0|>\delta\}$ we still have $\Egwhatdelta \equiv1$ and $\Ergwhatdelta= E_{\argwhat}$, independently of $\delta$.
		\end{enumerate}
	\end{theorem}
	\noindent\emph{Proof sketch for (a).}
	Our key step is to make the convex structure of the worst-case optimisation explicit. For any $f\in\scrF$, the function
	\begin{equation}\label{eq} f^c:\mathcal{X}\to\R;\quad \mu\mapsto f^c(\mu)=\inf_{{Q\in\scrP:\E_Q[X]=\mu}}\E_Q[f(X)]\,,\end{equation}
	is the largest convex function dominated by $f$, known as the \emph{convex envelope} of $f$, where the infimum is restricted to $Q$ for which $f$ is integrable. Combining this representation with admissibility of the coin-betting e-class gives
	$$\GROW(\calQ)=\sup_{\alpha\in I_{\mu_0}}\inf_{Q\in\calQ}\E_Q[\log(1+\alpha(X-\mu_0))]=\sup_{\alpha\in I_{\mu_0}}\inf_{Q\in\calQ}\E_Q[f_\alpha]=\sup_{\alpha\in I_{\mu_0}}f_\alpha^c(\mu_1)\,.$$
	By concavity of $f_\alpha$, it follows that $f_\alpha^c$ equals the linear interpolation of $f_\alpha(A)$ and $f_\alpha(B)$. That is, $f_\alpha^c(\mu_1) = \frac{B-\mu_1}{B-A}f_\alpha(A)+\frac{\mu_1-A}{B-A}f_\alpha(B)$, which is strictly concave in $\alpha$ with maximiser given by \eqref{eq:GROW_betting_parameter}.
	
	Since $\GRO(Q)=\sup_{\beta\in I_{\mu_0}}\E_Q[ f_\beta]$ for any $Q\in\calQ$, it follows similarly that 
	\begin{align}
		\REGROW(\calQ) &= \sup_{\alpha\in I_{\mu_0}}\inf_{Q\in\calQ}\inf_{\beta\in I_{\mu_0}} (\E_Q [f_\alpha]-\E_Q [f_\beta])= \sup_{\alpha\in I_{\mu_0}}\inf_{Q\in\calQ}\inf_{\beta\in I_{\mu_0}} \E_Q [\fab]\nonumber\\
		&=\sup_{\alpha\in I_{\mu_0}}\inf_{\beta\in I_{\mu_0}}\inf_{Q\in\calQ} \E_Q [\fab]=\sup_{\alpha\in I_{\mu_0}}\inf_{\beta\in I_{\mu_0}}\fabc(\mu_1)\,.
		\label{eq:proof_REGROW_123}
	\end{align}
	It remains to show that there exists a unique solution $\argw\in(\alpha^\star_{\GW}, \alphamax)$ to the maximin problem in~\eqref{eq:proof_REGROW_123}; see the Supplement for the details. The proofs for $(b)$ and $(c)$ follow similar ideas. \hfill$\square$
	
	\vspace{0.2cm}
	
	\begin{figure}[h]
		\centering
		\includegraphics[width =0.65\textwidth]{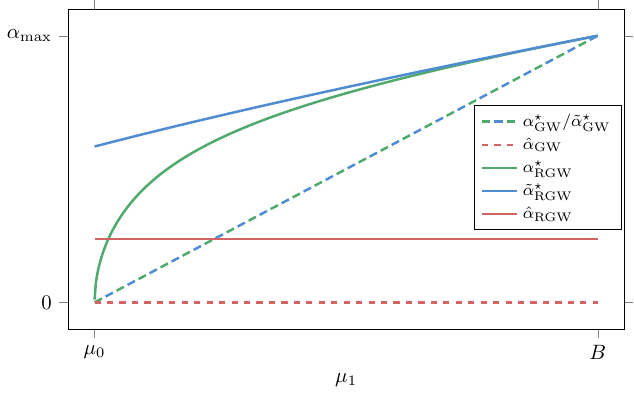}
		\caption{The optimal betting parameters from Theorem~\ref{thm:master_thm} as functions of $\mu_1$, corresponding to increasing separation between the null and the alternative from left to right. Notably, the optimal parameters in~(c) (red lines) are independent of the separation $\mu_1-\mu_0$.} 
		\label{fig:curves}
	\end{figure}

	Figure~\ref{fig:curves} illustrates the statement of Theorem~\ref{thm:master_thm}. In particular, for $\mu_1=\mu_0$ (no separation between the null and the alternative)  GROW yields a trivial solution, whereas REGROW remains non-trivial (i.e., the betting parameter is non-zero) for the one- and two-sided testing problems
	$\tilde{\calP}: \E(X)\leq \mu_0$ vs.\ $\tilde{\calQ}: \E(X)> \mu_0$, and ${\calP}: \E(X)=\mu_0$ vs.\ $\hat{\calQ}: \E(X)\neq \mu_0$. This is consistent with what was observed earlier for the Bernoulli example.
	Moreover, again as in the Bernoulli example, the optimal (RE)GROW e-variables for the two-sided problem are independent of the separation between the null and the alternative, with $\Ergwhat\not\equiv1$ whenever $\mu_0\neq(A+B)/2$.

	\begin{remark}The parameter $\argwprime$ at~\eqref{eq:alphaexplicit} is equivalently characterised as the unique solution to $ F_{\argwprime,\alphamax}(B)=G_{\argwprime,\agw}(\mu_1)$ with $\argwprime>\agw$.\end{remark}\begin{remark} The assumption $\mu_0< \mu_1$ in Theorem~\ref{thm:master_thm} is only made for notational convenience; see Theorem~\ref{thm:master_thm_reversed} in the Supplement for the symmetric version of Theorem~\ref{thm:master_thm} for
		$\mu_1< \mu_0$.
	\end{remark}

	From the proof of Theorem \ref{thm:master_thm} in the Supplement, we can identify the \emph{worst-case distributions} achieving the infimum over $\calQ$ in~\eqref{eq:def_GROW} and~\eqref{eq:def_REGROW}. For the point testing problem in Theorem~\ref{thm:master_thm}~(a), the infimum over $\calQ$ in GROW is achieved by $Q_1 = q\delta_A + (1-q)\delta_B$, with $q=(B-\mu_1)/(B-A)$, and $\Egw$ is the GRO e-variable for the simple alternative $\{Q_1\}$. This also admits a RIPr interpretation as in \citet{grunwald2024safe}: $Q_1$ minimises over $\calQ$ the reverse-KL distance to $\calP$, and $\Egw$ is the likelihood ratio between $Q_1$ and its projection $p\delta_A + (1-p)\delta_B$, with $p=(B-\mu_0)/(B-A)$. 
	
	For REGROW, the infimum is achieved simultaneously by $Q_1$ and $Q_2 = \delta_{\mu_1}$, which are the measures with highest and lowest variance in $\calQ$. As the worst-case distribution for GROW, $Q_1$ is the \emph{hardest} alternative to distinguish from $\calP$. On the other hand, the deterministic distribution $Q_2$  constitutes the \emph{easiest} alternative to test against $\calP$ and thus allows for the most powerful GRO e-variable $\Eg_{Q_2}=E_{\alphamax}$ over all $Q\in\calQ$:
	\begin{align*}
		\E_Q[\log\Eg_Q]&= \E_Q[\log(1+\alpha^\star_Q(X-\mu_0))]\leq \log \E_Q[1+\alpha^\star_Q(X-\mu_0)]\\
		&=\log [1+\alpha^\star_Q(\mu_1-\mu_0)]\leq \log [1+\alphamax(\mu_1-\mu_0)]=\E_{Q_2}[\log\Eg_{Q_2}]\,,
	\end{align*}
	where $\alpha^\star_Q$ denotes the GRO optimal betting parameter for the simple alternative $\{Q\}$. 
	Since REGROW evaluates performance relative to $\GRO(Q)$, alternatives with large $\GRO(Q)$ are also those for which a fixed e-variable may incur the largest relative loss, which is why $Q_2$ emerges as a worst-case distribution for REGROW despite being the easiest in absolute terms.
	
	For the one-sided testing problem in Theorem~\ref{thm:master_thm}~(b), the infimum in GROW is also achieved by $Q_1$ from above, although this time this is only a limiting distribution, in the sense that it is not contained in the alternative $\calQprime$. Again, for REGROW we have two worst-case distributions: one is $Q_1$, while the other is $\delta_B$. The reasoning that $\delta_B$ appears as a second worst-case candidate for REGROW is analogous to the argument above since it allows for the most powerful GRO e-variable amongst all measures in $\calQprime$. Finally, for the two-sided alternative in Theorem~\ref{thm:master_thm}~(c), we cannot speak of a specific worst-case distribution for GROW, since the GROW e-variable is trivial, whereas, for REGROW, the infimum is achieved simultaneously by $\delta_A$ and $\delta_B$. In contrast to the two other cases, both worst-case distributions depend only on the boundary of the support of the random variable and not on the specific choice of $\mu_0$, due to the convexity of the problem.
	
	We remark that in all cases the least favourable alternatives are supported on at most two points. 
	Indeed, for fixed parameters $\alpha$ and $\beta$, the map
	$
	Q \mapsto \E_Q[F_{\alpha,\beta}(X)]
	$
	is linear in $Q$. Hence, when minimising over a convex alternative, the infimum is attained (or approached) at an extreme point of the alternative. In our case, the alternatives are (unions) of the form $\{P\mid\E_P[X]=\mu\}$, and for each of these sets the extreme points are given precisely by the measures supported on at most two points (this is a direct consequence of Carathéodory's theorem in dimension $1$). For parts (b) and (c) of Theorem~\ref{thm:master_thm}, the fact that the worst-case distributions for REGROW are supported only on the extremal points $A$ and $B$ shows that they correspond exactly to those found in the introductory Bernoulli example when $A=0$ and $B=1$. 
	Notably, this is not the case for Theorem~\ref{thm:master_thm} (a), as there one of the worst-case distributions is not supported on $\{A,B\}$. 
	
	It is actually worth highlighting further structural similarities between Theorem~\ref{thm:master_thm} and the Bernoulli example. In all cases, the REGROW e-variable is obtained by equalising the relative expected growth under two worst-case distributions, corresponding to an increasing and a decreasing contribution, respectively. Moreover, as in the Bernoulli example, the resulting REGROW e-variables differ from the GRO e-variables for the respective worst-case distributions. Finally, we note that reductions of bounded random variables to Bernoulli variables are known in various other contexts; see, e.g.,~\citet[Lemma 3]{maurer2004a}.

	\section{Discussion}
	We have derived the (RE)GROW-optimal e-variables for bounded mean testing, a key example that does not fall under the AC assumption. Our proof combines the admissibility reduction to coin-betting e-variables with a convex-envelope reformulation of the  worst-case optimisation. For both criteria, we have identified the alternatives that are most difficult to test against and shown that they admit a meaningful interpretation. Remarkably, the interpretation of these worst-case distributions, as well as the reasoning underlying the derivation of the (RE)GROW optimal e-variables in this study, align with the much simpler introductory Bernoulli example. We hope that our findings shed light on (RE)GROW beyond our setting and contribute to the current discussion on how to find optimal e-variables for composite alternatives (\citealp{wang2024proposer}, \citealp{larsson2025the}). 
	
	It is natural to ask to what extent our results generalise beyond the ubiquitous problem of bounded mean testing. Importantly, the characterisation of worst-case expectations in terms of convex envelopes, as in~\eqref{eq:deffc}, directly extends to multivariate functions; see, e.g., \cite{rockafellar1970convex}. Combined with suitable finite-dimensional admissible e-classes, this provides a potential strategy for studying more general hypotheses defined by linear constraints, such as quantile testing or testing the mean under a finite-variance assumption~\citep{clerico2024optimal,larsson2025evariables}.
	For future work, it would be of interest to investigate whether, in concrete settings, it is possible to characterise the corresponding convex envelopes and to solve the associated (mini)max problems explicitly. For GROW, this is likely to be feasible, since the relevant convex envelope is linear due to the concavity of the objective function appearing in the worst-case expectation. For REGROW, an explicit characterisation of the corresponding convex envelope is probably  more challenging in general, and some of the monotonicity arguments we used  would need to be adapted to a  no longer totally ordered parameter space. Finally, we note that the convex-envelope reformulation itself is not specific to logarithmic utility and also applies to other increasing concave utilities, including the power utilities studied by~\citet{larsson2025the}.
	
	A limitation of our work is its restriction to single-round e-variables. Notably, the optimality of the coin-betting e-class directly extends to the multi-round game with $n$ outcomes $x_1,\ldots, x_n$, and optimal e-variables are of the form $\prod_{i=1}^n (1+\alpha_i(x_i-\mu_0))$; see \cite{clerico2025on}. However, for the multi-round setting optimal e-variables are parametrised by predictable betting strategies rather than by a single deterministic parameter, making the corresponding (RE)GROW problems substantially less tractable. A distinct line of work constructs adaptive strategies that do not prespecify an alternative and instead achieve asymptotic, instance-wise log-optimality. We discuss the relation between these approaches and (RE)GROW in  the Supplement. A natural direction for future work is to study exact or approximate (RE)GROW strategies in the multi-round setting, including their possible connection with universal-portfolio and mixture constructions.
	
	\section*{Acknowledgement}
	We thank Wouter Koolen, Peter Gr\"unwald, and Nick Koning for helpful discussions. We also thank the three anonymous referees, whose comments helped improve the quality of the paper.
	
	\bibliographystyle{plainnat}
	\bibliography{bib}

@misc{clerico2024optimal,
	author        = {Clerico, Eugenio},
	title         = {Optimal e-value testing for properly constrained hypotheses},
	year          = {2024},
	note          = {arXiv:2412.21125},
	eprint        = {2412.21125},
	archivePrefix = {arXiv},
	primaryClass  = {math.ST},
	url           = {https://arxiv.org/abs/2412.21125}
}

@article{clerico2025on,
	author  = {Clerico, Eugenio},
	title   = {On the optimality of coin-betting for mean estimation},
	journal = {Int. J. Approx. Reason.},
	volume  = {187},
	pages   = {109550},
	year    = {2025},
	doi     = {10.1016/j.ijar.2025.109550}
}

@book{cover2006elements,
	author    = {Cover, Thomas M. and Thomas, Joy A.},
	title     = {Elements of Information Theory},
	edition   = {2},
	publisher = {John Wiley \& Sons},
	address   = {Hoboken, New Jersey},
	year      = {2006},
	doi       = {10.1002/047174882X}
}

@article{grunwald2024safe,
	author  = {Gr{\"u}nwald, Peter D. and de Heide, Rianne and Koolen, Wouter M.},
	title   = {Safe testing},
	journal = {J. R. Statist. Soc. B},
	volume  = {86},
	number  = {5},
	pages   = {1091--1128},
	year    = {2024},
	doi     = {10.1093/jrsssb/qkae011}
}

@article{kelly1956a,
	author  = {Kelly, Jr., John L.},
	title   = {A new interpretation of information rate},
	journal = {Bell Syst. Tech. J.},
	volume  = {35},
	number  = {4},
	pages   = {917--926},
	year    = {1956},
	doi     = {10.1002/j.1538-7305.1956.tb03809.x}
}

@misc{larsson2025evariables,
	author        = {Larsson, Martin and Ramdas, Aaditya and Ruf, Johannes},
	title         = {E-variables for hypotheses generated by constraints},
	year          = {2025},
	note          = {arXiv:2504.02974},
	eprint        = {2504.02974},
	archivePrefix = {arXiv},
	primaryClass  = {math.ST},
	url           = {https://arxiv.org/abs/2504.02974}
}

@article{larsson2025the,
	author  = {Larsson, Martin and Ramdas, Aaditya and Ruf, Johannes},
	title   = {The numeraire e-variable and reverse information projection},
	journal = {Ann. Statist.},
	volume  = {53},
	number  = {3},
	pages   = {1015--1043},
	year    = {2025},
	doi     = {10.1214/24-AOS2487}
}

@misc{maurer2004a,
	author        = {Maurer, Andreas},
	title         = {A note on the {PAC Bayesian} theorem},
	year          = {2004},
	note          = {arXiv:cs/0411099},
	eprint        = {cs/0411099},
	archivePrefix = {arXiv},
	primaryClass  = {cs.LG},
	url           = {https://arxiv.org/abs/cs/0411099}
}

@article{orabona2024tight,
	author  = {Orabona, Francesco and Jun, Kwang-Sung},
	title   = {Tight concentrations and confidence sequences from the regret of universal portfolio},
	journal = {IEEE Trans. Inform. Theory},
	volume  = {70},
	number  = {1},
	pages   = {436--455},
	year    = {2024},
	doi     = {10.1109/TIT.2023.3330187}
}

@article{ramdas2023gametheoretic,
	author  = {Ramdas, Aaditya and Gr{\"u}nwald, Peter D. and Vovk, Vladimir and Shafer, Glenn},
	title   = {Game-theoretic statistics and safe anytime-valid inference},
	journal = {Statist. Sci.},
	volume  = {38},
	number  = {4},
	pages   = {576--601},
	year    = {2023},
	doi     = {10.1214/23-STS894}
}

@article{ramdas2025hypothesis,
	author  = {Ramdas, Aaditya and Wang, Ruodu},
	title   = {Hypothesis testing with e-values},
	journal = {Found. Trends Stat.},
	volume  = {1},
	number  = {1--2},
	pages   = {1--390},
	year    = {2025},
	doi     = {10.1561/3600000002}
}

@book{rockafellar1970convex,
	author    = {Rockafellar, R. Tyrrell},
	title     = {Convex analysis},
	series    = {Princeton Mathematical Series},
	volume    = {28},
	publisher = {Princeton University Press},
	address   = {Princeton, New Jersey},
	year      = {1970}
}

@article{shafer2021testing,
	author  = {Shafer, Glenn},
	title   = {Testing by betting: A strategy for statistical and scientific communication},
	journal = {J. R. Statist. Soc. A},
	volume  = {184},
	number  = {2},
	pages   = {407--431},
	year    = {2021},
	doi     = {10.1111/rssa.12647}
}

@article{vovk2021evalues,
	author  = {Vovk, Vladimir and Wang, Ruodu},
	title   = {E-values: Calibration, combination and applications},
	journal = {Ann. Statist.},
	volume  = {49},
	number  = {3},
	pages   = {1736--1754},
	year    = {2021},
	doi     = {10.1214/20-AOS2020}
}

@article{wang2026ebacktesting,
	author  = {Wang, Qiuqi and Wang, Ruodu and Ziegel, Johanna},
	title   = {E-backtesting},
	journal = {Manag. Sci.},
	volume  = {72},
	number  = {6},
	pages   = {4952--4973},
	year    = {2026},
	doi     = {10.1287/mnsc.2023.01659}
}

@article{wang2024proposer,
	author  = {Wang, Ruodu},
	title   = {{Proposer of the vote of thanks to Gr{\"u}nwald, de Heide, and Koolen and contribution to the discussion of ``Safe testing''}},
	journal = {J. R. Statist. Soc. B},
	volume  = {86},
	number  = {5},
	pages   = {1129--1131},
	year    = {2024},
	doi     = {10.1093/jrsssb/qkae071}
}

@article{waudbysmith2024estimating,
	author  = {Waudby-Smith, Ian and Ramdas, Aaditya},
	title   = {Estimating means of bounded random variables by betting},
	journal = {J. R. Statist. Soc. B},
	volume  = {86},
	number  = {1},
	pages   = {1--27},
	year    = {2024},
	doi     = {10.1093/jrsssb/qkad009}
}

@misc{waudbysmith2025universal,
	author        = {Waudby-Smith, Ian and Sandoval, Ricardo and Jordan, Michael I.},
	title         = {Universal log-optimality for general classes of e-processes and sequential hypothesis tests},
	year          = {2025},
	note          = {arXiv:2504.02818},
	eprint        = {2504.02818},
	archivePrefix = {arXiv},
	primaryClass  = {stat.ME},
	url           = {https://arxiv.org/abs/2504.02818}
}
	
	\clearpage
	
	\begin{center}
		{\LARGE\bfseries Supplementary Material}
	\end{center}
	
	\appendix
	
	\section{Visual intuition for the Bernoulli example}
	\begin{figure}[H]
		\centering
		\includegraphics[width =0.65\textwidth]{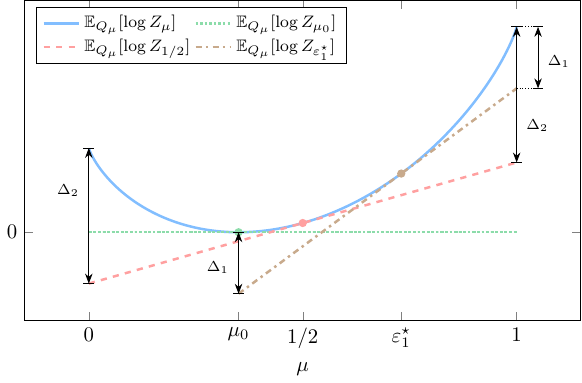}
		\caption{The solid blue curve shows the optimal growth rate $\GRO(Q_\mu)=\E_{Q_\mu}[\log Z_\mu]=\kl(\mu,\mu_0)$ as a function of $\mu$.
			For any fixed $\ee\in[0,1]$, the e-power $\E_{Q_\mu}[\log Z_\ee]$ is affine in $\mu$ and coincides with the tangent to the blue curve at $\mu=\ee$.
			The green dotted line corresponds to the trivial e-variable $Z_{\mu_0}=(1,1)^\top$, which has zero e-power for all $\mu$.
			It is the only tangent that is everywhere nonnegative, hence $Z_{\mu_0}$ is the GROW solution for both $\calQ_1$ and $\calQ_2$.
			For REGROW, the relevant quantity is the gap between $\mu\mapsto\GRO(Q_\mu)$ and its tangents.
			For $\calQ_1$, only $\mu\in(\mu_0,1]$ matter, and the optimal e-variable $Z_{\ee_1^\star}$ (brown line) equalises this gap at $\mu=\mu_0$ and $\mu=1$ ($\Delta_1$).
			For $\calQ_2$, the REGROW solution is $Z_{1/2}$ (red dashed line), for which the worst cases are $\mu=0$ and $\mu=1$, where the gap to $\GRO(Q_\mu)$ equals $\Delta_2$.
		} 
		\label{fig:example}
	\end{figure}
	
	\section{Admissibility of the coin-betting e-class}
	Recall that, for fixed $\mu_0\in(A,B)$, we defined the \emph{coin-betting e-variable} as
	\begin{equation*}
		E_\alpha\,:\,x\mapsto 1+\alpha (x-\mu_0)
	\end{equation*}
	for $I_{\mu_0}= [(\mu_0-B)^{-1},(\mu_0-A)^{-1}]$. 
	\cite{clerico2025on} shows that the coin-betting e-variables are the only admissible e-variables for $\calP = \{P\in\scrP \mid \E_P[X] = \mu_0\}$, in the sense that any other e-variable is pointwise dominated by a coin-betting e-variable. A similar result had  earlier appeared in \cite{wang2026ebacktesting} (cf.\ Lemma E.2 therein), while \cite{clerico2024optimal} and \cite{larsson2025evariables} established extensions to more general hypotheses generated by linear constraints. Since we repeatedly rely on the admissibility of the coin-betting class throughout the paper, we restate the main result of \citet{clerico2025on} here for the reader's convenience.
	\begin{proposition}[From Theorem 1 in \citealp{clerico2025on}]\label{prop:coin_betting_is_an_optimal_class}
		Let $\calP = \{P\in\scrP \mid \E_P[X] = \mu_0\}$. For any $E\in\calE(\calP)$, there exists $\alpha\in I_{\mu_0}$ such that $E_\alpha(x)\geq E(x)$, for all $x\in\calX.$
	\end{proposition}

	For the proof of Theorem~\ref{thm:master_thm}~(b), we use the following auxiliary result on one-sided mean testing. 
	
	\begin{corollary}\label{cor:onesided}
		For $\mu_0\in(A,B)$, consider
		$\tilde{\calP} = \{P\in\scrP \mid\E_P[X]\leq\mu_0\}$.
		Then, for every $E\in\calE(\tilde{\calP})$, there exists $\alpha\in [0,\alphamax]$ such that $E_{\alpha}(x)\geq E(x)$, for all $x\in\calX$.
	\end{corollary}
	\begin{proof}
		For $\mu\in[A,B]$, let $\calP_\mu = \{P\in\scrP\mid\E_P[X]=\mu\}$. Then we have
		$$\tilde{\calP} = \{P\in\scrP\mid\E_P[X]\leq\mu_0\} = \bigcup_{A\leq\mu\leq\mu_0}\calP_\mu\,.$$
		As discussed in \cite{larsson2025evariables}, it is straightforward to check that the set of e-variables for a union of hypotheses is the intersection of the sets of e-variables for each hypothesis, namely
		$$\calE(\tilde{\calP}) = \bigcap_{A\leq\mu\leq\mu_0}\calE(\calP_\mu)\,.$$
		Clearly, $\calE(\tilde{\calP})\subseteq \calE(\calP_{\mu_0})$. From Proposition \ref{prop:coin_betting_is_an_optimal_class}, it is easy to check that, if $\alpha\in[0,\alphamax]$, then $E_\alpha^\mu\,:\,x\mapsto 1+\alpha(x-\mu)$ is an e-variable for $\calP_\mu$, for every $\mu\in[A,\mu_0]$. (Note that although $\mu=A$ is not explicitly covered by Proposition \ref{prop:coin_betting_is_an_optimal_class}, since $\calP_A = \{\delta_A\}$, we immediately see that any measurable function $E$ such that $E(A)\leq 1$ is in $\calE(\calP_A)$.) Moreover, in such case we also have that $E_\alpha^\mu\geq E_\alpha^{\mu_0}$. In particular, if $E\in\calE(\calP_{\mu_0})$ is dominated by $E_\alpha^{\mu_0}$, for $\alpha\in[0,\alphamax]$, we have that $E\in\calE(\tilde{\calP})$.  On the other hand, if $E\in\calE(\calP_{\mu_0})$ is not dominated by $E_\alpha^{\mu_0}$ for any $\alpha\geq 0$, then there is $x\in[A,\mu_0]$ such that $E(x)>1$. Since $\delta_x\in\tilde{\calP}$, we have that $E\notin\calE(\tilde{\calP})$, and so we conclude.
	\end{proof}

	\section{C-envelopes}\label{app:convex_envelope}
	Given $f\in\scrF$, we define the \emph{c-envelope of $f$} as the mapping $f^c:\calX\to[-\infty,\infty)$ given by
	\begin{equation}\label{eq:deffc}   f^c(x) = \inf_{{Q\in\scrP\,:\E_Q[X]=x}}\E_Q[f(X)]\,,\qquad x\in \calX\,,\end{equation}
	with the infimum restricted to those $Q$ for which $f$ is integrable. 
	\begin{lemma}\label{lemma:fc}
		Let $f\in\scrF$ and  $f^c$ be its c-envelope \eqref{eq:deffc}. Then, $f^c$ is dominated by $f$, that is $f^c\leq f$ pointwise. If $f$ is non-increasing (non-decreasing), then $f^c$ is non-increasing (non-decreasing). If $f$ is convex, then $f^c= f$, while if $f$ is concave we have that $f^c(x) = \frac{B-x}{B-A}f(A)+\frac{x-A}{B-A} f(B)$ for all $x$, with the convention $0/0=1$. Moreover, again with the convention $0/0=1$, we have the alternative characterisation
		\begin{equation}\label{eq:characterization_convex_envelope}
			f^c(x)=   \inf_{\substack{x_1,x_2 \in \calX \\ x_1\leq x\leq x_2}} \left\{\frac{x_2-x}{x_2-x_1}\,f(x_1)+\left(1-\frac{x_2-x}{x_2-x_1}\right)f(x_2)\right\}\,.
		\end{equation}
	\end{lemma}
	Readers familiar with convex analysis will recognise that $f^c$ is the \emph{convex envelope} of $f$, namely its \emph{greatest convex minorant}. Both characterisations~\eqref{eq:deffc} and~\eqref{eq:characterization_convex_envelope}, as well as the other properties stated in the lemma, are standard; see, e.g., Chapter~12 of~\citet{rockafellar1970convex}. Yet, our analysis relies only on the facts explicitly stated in Lemma~\ref{lemma:fc}. For this reason, we avoid invoking additional tools or notions from convex analysis, and give a self-contained proof of Lemma~\ref{lemma:fc}.
	
	\begin{proof}
		
		The fact that $f^c$ is dominated by $f$ is trivial since $\delta_x\in\scrP$ has mean $x$.
		
		Assume that $f$ is non-decreasing (the non-increasing case is analogous). Fix $x_1<x_2$ in $\calX=[A,B]$, and let $\lambda = (x_1-A)/(x_2-A)\in[0,1)$. Fix any $Q_2\in\scrP$ such that $\E_{Q_2}[X] = x_2$, and define $Q_1 = \lambda Q_2 + (1-\lambda)\delta_A$. Then $Q_1\in\scrP$ and $\E_{Q_1}[X] = x_1$. We have that $\E_{Q_1}[f] = \lambda \E_{Q_2}[f] + (1-\lambda)f(A)\leq \E_{Q_2}[f]$, since $f(A) = \min f \leq \E_{Q_2}[f]$ because $f$ is non-decreasing. We conclude that $f^c(x_1)\leq f^c(x_2)$. 
		
		If $f$ is convex, then by Jensen's inequality we have $f^c\geq f$, and so $f^c\equiv f$. If $f$ is concave, let $\ell\,:\,x\mapsto \frac{B-x}{B-A}f(A)+(1-\frac{B-x}{B-A}) f(B)$. Then $f\geq \ell$ by concavity, and $\ell(x)\geq f^c(x)$ for all $x$, since it is realised by the measure $\frac{B-x}{B-A}\delta_A + (1-\frac{B-x}{B-A})\delta_B$. For any measure $Q$ with mean $x$, by linearity of $\ell$ we have $\E_Q[f]\geq\E_Q[\ell] = \ell(\E_Q[X]) = \ell(x)$, so $\ell\equiv f^c$. 
		
		We are left with the final characterisation \eqref{eq:characterization_convex_envelope}. The fact that the RHS is larger than the LHS is obvious, so only the other direction needs to be shown. Fix $x\in\calX$. If $f^c(x) = f(x)$ then this is achieved by \eqref{eq:characterization_convex_envelope} with $x_1=x_2=x$. So let us assume that $f^c(x)<f(x)$. Note that in such case $x$ must be in the interior of $\calX$. We can find a measure $Q\in\scrP$, with mean $x$, such that $\E_Q[f]<f(x)$. To conclude, it is enough to show that there are $y,z\in\calX$, with $y<x<z$, such that $\E_{\theta_{y,z}}[f]\leq \E_Q[f]$, where $\theta_{y,z}$ is the only probability measure with mean $x$ supported on $\{y,z\}$.  Without loss of generality, we can assume that $Q$ puts no mass on $x$. (If $Q(\{x\})>0$, write $Q=\lambda Q' + (1-\lambda)\delta_x$ with $\lambda\in(0,1)$ and $Q'(\{x\})=0$. Then $\E_{Q'}[X]=x$ and $\E_{Q'}[f]< \E_Q[f]$, so we may replace $Q$ by $Q'$.) Now, we have that
		\begin{align*}\int_{z> x}&\int_{y<x}(z-y)\big(\E_{\theta_{y,z}}[f]-\E_Q[f]\big) \dd Q(y)\dd Q(z) \\&= \int_{z>x}\int_{y<x}\big((z-x)(f(y)-\E_Q[f]) + (x-y)(f(z)-\E_Q[f])\big)\dd Q(y)\dd Q(z)\\&=\left(\int_{z>x} (z-x)\dd Q(z)\right)\left(\int_{y<x}(f(y)-\E_Q[f])\dd Q(y) + \int_{z> x}(f(z)-\E_Q[f])\dd Q(z)\right) \\&= \left(\int_{z>x} (z-x)\dd Q(z)\right)(\E_Q[f]-\E_Q[f])=0\,,\end{align*}
		where we used that $Q$ has no atom on $x$ and that $\int_{z> x} (z-x)\dd Q(z) = \int_{y< x} (x-y)\dd Q(y)$, because $x$ is the mean of $Q$. Since we are integrating over a domain where $z-y>0$, this integral being null implies that we cannot always have $\E_{\theta_{y,z}}[f]-\E_Q[f]>0$ for every $y$ and $z$. So, we conclude. To be fully rigorous, one should observe that the double integral manipulation only makes sense if $\E_Q[|f|]<+\infty$, which excludes the cases where $f$ is integrable but $\E_Q[f]=+\infty$ or $\E_Q[f]=-\infty$. The first case is trivial, as any allowed $\theta_{y,z}$ has finite expectation. For the case $-\infty$, we can define $f_n = \max{f,-n}$. We can then apply the previous result for each of these and find a sequence of two-point measures $\theta_n$ such that $\E_{\theta_n}[f_n]\leq \E_Q[f_n]$. Then we get
		\begin{equation*}
			\inf_{n}\E_{\theta_n}[f]\leq \inf_{n}\E_{\theta_n}[f_n]\leq \inf_{n}\E_{Q}[f_n] = -\infty\,,
		\end{equation*}
		which is enough to show that the RHS of \eqref{eq:characterization_convex_envelope} is $-\infty$, and hence conclude.
	\end{proof}

	\section{Proof of Theorem 
		\ref{thm:master_thm}}
	For $\alpha,\beta\in I_{\mu_0}$, we recall the following functions from $[A,B]$ to $[-\infty,\infty)$:
	\begin{align*}
		f_\alpha(x) &= \log\big(1+\alpha(x-\mu_0)\big)\,;\\
		\fab(x) &= \log\left(\frac{1+\alpha(x-\mu_0)}{1+\beta(x-\mu_0)}\right)\,;\\
		\gab(x) & = \frac{B-x}{B-A}\,F_{\alpha,\beta}(A) + \left(1-\frac{B-x}{B-A}\right)F_{\alpha,\beta}(B)\,,
	\end{align*}
	with the conventions $\log0=-\infty$ and $0/0=1$. The following facts are easy to verify: $f_\alpha$ is strictly concave for every $\alpha\neq 0$ (it is simply the constant $0$ if $\alpha=0$); $\fab$ is strictly increasing if $\alpha>\beta$ and strictly decreasing if $\beta>\alpha$ (again, the case $\alpha=\beta$ is trivial, as $\fab\equiv 0$); and $\fab$ 
	is concave for $0\le\beta\le\alpha$ and $\alpha\le\beta\le0$, and convex for $0\le\alpha\le\beta$ and $\beta\le\alpha\le0$. The monotonicity of $\gab$ clearly follows from the one of $\fab$.
	
	\begin{proof}[of Theorem~\ref{thm:master_thm}~(a)]
		It is easily checked that $\agw\in I_{\mu_0}$. Proposition \ref{prop:coin_betting_is_an_optimal_class} implies that $ \Egw$ can only be of the form $x\mapsto 1+\alpha^\star(x-\mu_0)$, for some $\alpha^\star\in I_{\mu_0}$. We need to check that this $\alpha^\star$ is precisely $\agw$ given at \eqref{eq:GROW_betting_parameter}. Recalling the definition \eqref{eq:deffc} of the c-envelope, by Proposition \ref{prop:coin_betting_is_an_optimal_class}, we have
		$$
		\GROW(\calQ) = \sup_{\alpha\in I_{\mu_0}}\inf_{Q\in\calQ} \E_Q [f_\alpha] = \sup_{\alpha\in I_{\mu_0}}f_{\alpha}^c(\mu_1)\,. 
		$$
		By Lemma \ref{lemma:fc}, the concavity of $f_\alpha$ implies that $f_\alpha^c(\mu_1) = \frac{B-\mu_1}{B-A}f_\alpha(A)+\frac{\mu_1-A}{B-A}f_\alpha(B)$. This is a strictly concave differentiable function of $\alpha$, whose maximiser is given by \eqref{eq:GROW_betting_parameter} (this can be easily derived by setting the derivative equal to zero).
		
		For the existence of a unique $\argw\in(\alpha^\star_{\GW}, \alphamax)$ satisfying $F_{\argw,\alphamax}(\mu_1)=G_{\argw,\agw}(\mu_1)$, it is enough to note that $F_{\alpha^\star_{\GW},\alphamax}(\mu_1)<G_{\alpha^\star_{\GW},\alpha^\star_{\GW}}(\mu_1)$, $F_{\alphamax,\alphamax}(\mu_1)>G_{\alphamax,\alpha^\star_{\GW}}(\mu_1)$, $\alpha \mapsto F_{\alpha,\alphamax}(\mu_1)$ is continuous and strictly increasing on $I_{\mu_0}$, and $\alpha \mapsto G_{\alpha,\alpha^\star_{\GW}}(\mu_1)$ is continuous and strictly decreasing for $\alpha > \alpha^\star_{\GW}$.
		
		Again by Proposition \ref{prop:coin_betting_is_an_optimal_class}, for every $Q\in\calQ$, $\GRO(Q)=\sup_{\beta\in I_{\mu_0}}\E_Q[ \log(1+\beta (X-\mu_0))]$. By definition \eqref{eq:deffc} of the c-envelope,
		\begin{align}
			\REGROW(\calQ) &= \sup_{\alpha\in I_{\mu_0}}\inf_{Q\in\calQ}\inf_{\beta\in I_{\mu_0}} \E_Q [\fab]\nonumber\\
			&=\sup_{\alpha\in I_{\mu_0}}\inf_{\beta\in I_{\mu_0}}\inf_{Q\in\calQ} \E_Q [\fab]\nonumber\\
			&=\sup_{\alpha\in I_{\mu_0}}\inf_{\beta\in I_{\mu_0}}\fabc(\mu_1)\,.
			\label{eq:proof_REGROW_1}
		\end{align}
		
		Define  $\varphi: I_{\mu_0}\to\R$ as $\varphi(\alpha) = \inf_{\beta\in I_{\mu_0}}\fabc(\mu_1)$. For $\alpha<\argw$,
		$$
		\phi(\alpha) \leq F_{\alpha,\alphamax}^c(\mu_1) \leq F_{\alpha,\alphamax}(\mu_1) \leq F_{\argw,\alphamax}(\mu_1)\,,
		$$
		where we used Lemma \ref{lemma:fc} and the monotonicity of $\alpha \mapsto F_{\alpha,\alphamax}(\mu_1)$ for the second and third inequality. For $\alpha>\beta>0$, $\fab$ is concave and $\fabc\equiv \gab$ by Lemma \ref{lemma:fc}. Hence, for $\alpha>\argw$,
		$$
		\phi(\alpha) \leq F_{\alpha,\alpha_{\GW}^\star}^c(\mu_1) = G_{\alpha, \alpha_{\GW}^\star}(\mu_1) \leq G_{\argw, \alpha_{\GW}^\star}(\mu_1) =F_{\argw,\alphamax}(\mu_1)\,,
		$$
		where we used the monotonicity of $\alpha\mapsto G_{\alpha,\agw}(\mu_1)$ in the second inequality. Therefore, $F_{\argw,\alphamax}(\mu_1)$ globally upper bounds $\phi$. To conclude, it remains to show that this upper bound is actually achieved by $\argw$, namely
		\begin{equation}\label{eq:proof_REGROW_eq3}
			\phi(\argw) = \inf_{\beta\in I_{\mu_0}}F_{\argw, \beta}^c(\mu_1)= F_{\argw,\alphamax}(\mu_1)\,.
		\end{equation}
		We proceed by case distinction. First, consider  $\beta\geq\argw$. Then, $F_{\argw, \beta}$ is convex, which by Lemma \ref{lemma:fc} implies $ F_{\argw, \beta}\equiv F_{\argw, \beta}^c$. Therefore,
		$$
		\inf_{\beta\geq\argw}F_{\argw, \beta}^c(\mu_1) = \inf_{\beta\geq\argw}F_{\argw, \beta}(\mu_1)= F_{\argw,\alphamax}(\mu_1)\,,
		$$
		by monotonicity of $\beta \mapsto F_{\argw,\beta}(\mu_1)$.
		
		On the other hand, for $\beta\in(0,\argw)$, $ F_{\argw, \beta}$ is concave. So, $ F_{\argw, \beta}^c\equiv G_{\argw, \beta}$ by Lemma \ref{lemma:fc} and
		$$
		\inf_{\beta\in (0,\argw)}F_{\argw, \beta}^c(\mu_1) = \inf_{\beta\in (0,\argw)}G_{\argw, \beta}(\mu_1)= G_{\argw, \alpha^\star_{\GW}}(\mu_1)= F_{\argw,\alphamax}(\mu_1)\,,
		$$
		where we used that $\beta \mapsto G_{\argw,\beta}(\mu_1)$ has a unique minimiser in $I_{\mu_0}$ at $\beta=\agw\in(0,\argw)$.
		
		To conclude the proof, it remains to consider the case $\beta \leq 0$. We will show that $\beta \mapsto F_{\argw, \beta}^c(\mu_1)$ is decreasing for $\beta\leq 0$.  By Lemma \ref{lemma:fc}, we have
		\begin{equation*}
			F_{\argw, \beta}^c(\mu_1) = \inf_{A\leq x_1 \leq \mu_1 \leq x_2\leq B} \big\{\theta{(x_1,x_2)}\, F_{\argw, \beta}(x_1)+(1-\theta(x_1,x_2))\, F_{\argw, \beta}(x_2)\big\}\,,
		\end{equation*}
		with $\theta{(x_1,x_2)}={(x_2-\mu_1)}/{(x_2-x_1)}$. For any fixed $x_1 \leq \mu_1 \leq x_2$, define $L:[\alphamin,0]\to\R$ via 
		\begin{equation*}
			L(\beta)=\theta{(x_1,x_2)}\, F_{\argw, \beta}(x_1)+(1-\theta{(x_1,x_2)})\, F_{\argw, \beta}(x_2) \,.
		\end{equation*}
		$L$ is differentiable, with first derivative given by
		$$
		L'(\beta) = -\frac{x_2-\mu_1}{x_2-x_1} \cdot\frac{x_1-\mu_0}{1+\beta (x_1-\mu_0)}-\frac{\mu_1-x_1}{x_2-x_1} \cdot\frac{x_2-\mu_0}{1+\beta (x_2-\mu_0)}\,. 
		$$
		If $x_1\geq \mu_0,$ then $L'(\beta)$ is clearly negative. The same holds true even for $x_1\in[A, \mu_0)$, as one can quickly check that
		\begin{equation*}
			\frac{\mu_1-x_1}{\mu_0-x_1}\geq 1 \geq \frac{x_2-\mu_1}{x_2-\mu_0}\cdot \frac{1+\beta (x_2-\mu_0)}{1+\beta (x_1-\mu_0)}\,,
		\end{equation*}
		(as both fractions on the right hand side are strictly smaller than $1$ for $x_1< \mu_0$ and $\beta<0$), which by simple algebraic manipulation is equivalent to $L'(\beta)\leq 0$. We conclude that $\beta \mapsto F_{\argw, \beta}^c(\mu_1)$ is decreasing for $\beta \leq 0$, as it is the infimum of a family of decreasing functions.  The conclusion follows.
	\end{proof}

	\begin{figure}[t]
		\centering
		\includegraphics[width =0.5\textwidth]{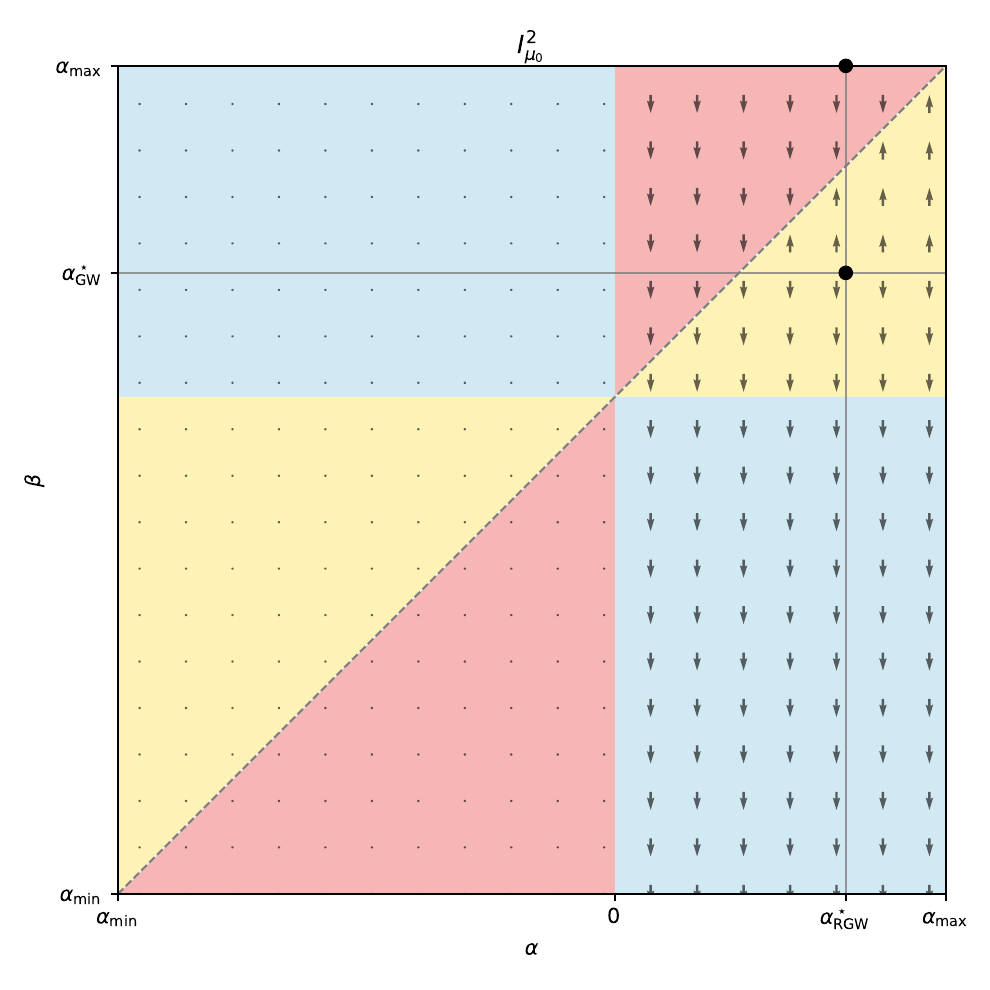}
		\caption{The parameter space $I_{\mu_0}^2$ 
			with the regions where $\fab$ is convex (red) or concave (yellow), for $\mu_1>\mu_0$. By Lemma \ref{lemma:fc}, $\fabc\equiv\fab$ in the red area and $\fabc\equiv\gab$ in the yellow area. 
			For $\alpha>0$, the directional derivatives of the function $\beta\mapsto \fabc(\mu_1)$ are represented as  arrows. Note that $\beta\mapsto F^c_{\argw,\beta}(\mu_1)$ has two local minima, achieved at the two black dots in the figure. Formally proving this fact is the main step to show \eqref{eq:proof_REGROW_eq3}. In particular, the two black dots are the two unique points simultaneously achieving the sup-inf in \eqref{eq:proof_REGROW_1}.} 
		\label{fig:Regions_F}
	\end{figure}

	Figure \ref{fig:Regions_F} gives a graphical illustration of the proof of Theorem \ref{thm:master_thm} (a). 
	\begin{proof}[of Theorem~\ref{thm:master_thm}~(b)]
		By Corollary \ref{cor:onesided} and Lemma \ref{lemma:fc},
		\begin{equation}\label{eq:proof_GROW_eq1_one-sided}
			\GROW(\calQprime) = \sup_{0\leq\alpha\leq\alphamax}\inf_{Q\in\calQprime} \E_Q[ \log(1+\alpha (X-\mu_0))] = \sup_{0\leq\alpha\leq\alphamax} \inf_{B\geq\mu > \mu_1} f_{\alpha}^c(\mu)\,.
		\end{equation}
		In the proof of Theorem~\ref{thm:master_thm}~(a), it is shown that, for any $\mu \in [\mu_1,B]$, the function $\alpha \mapsto f_\alpha^c(\mu)$ is strictly concave with unique positive maximiser $\frac{(\mu-\mu_0)}{(\mu_0-A)(B-\mu_0)} \in [0,\alphamax]$. 
		For any $\alpha\geq 0$, the function $\mu \mapsto f_{\alpha}^c(\mu)$ is non-decreasing. By continuity  the infimum in \eqref{eq:proof_GROW_eq1_one-sided} equals the limiting value at $\mu=\mu_1$.
		
		We now define $\argwprime$ as the unique element of $(\agw,\alphamax)$ satisfying $F_{\argwprime,\alphamax}(B)=G_{\argwprime,\agw}(\mu_1)$. Existence and uniqueness follow from the facts that $F_{\alpha^\star_{\GW},\alphamax}(B)<G_{\alpha^\star_{\GW},\alpha^\star_{\GW}}(\mu_1)$, $F_{\alphamax,\alphamax}(B)>G_{\alphamax,\alpha^\star_{\GW}}(\mu_1)$, $\alpha\mapsto F_{\alpha,\alphamax}(B)$ is continuous and strictly increasing, and $\alpha\mapsto G_{\alpha,\alpha^\star_{\GW}}(\mu_1)$ is continuous and strictly decreasing for $\alpha>\alpha^\star_{\GW}$. That this solution is given by~\eqref{eq:alphaexplicit} can be checked by direct calculation.

		It follows analogously to the proof of part (a) that 
		\begin{align}
			\REGROW(\calQprime) =\sup_{\alpha\in I_{\mu_0}^+}\inf_{\beta\in I_{\mu_0}^+}\inf_{\mu > \mu_1}\fabc(\mu)
			\label{eq:proof_REGROW_1-one-sided}
		\end{align}
		where $I_{\mu_0}^+=[0,\alphamax]$, and where \eqref{eq:proof_REGROW_1-one-sided} differs from \eqref{eq:proof_REGROW_1} only by the additional infimum over $\mu\geq \mu_1$ (noting that taking the infimum over $\mu\geq \mu_1$ and $\mu>\mu_1$ is equivalent by continuity). 
		
		For
		\begin{equation*}
			\varphi: I_{\mu_0}^+\to \bar{\R}; \quad \alpha \mapsto \inf_{\beta\in I_{\mu_0}^+}\inf_{\mu \geq \mu_1}\fabc(\mu),
		\end{equation*}
		it follows analogously to the proof of part (a) that $F_{\argwprime,\alphamax}(B)$ globally upper bounds $\phi$, and it remains to show that this upper bound is actually achieved by $\argwprime$, that is
		\begin{equation}\label{eq:proof_REGROW_eq3-one-sided}
			\phi(\argwprime) = \inf_{\beta\in I_{\mu_0}^+}\inf_{\mu \geq \mu_1}F_{\argwprime, \beta}^c(\mu)= F_{\argwprime,\alphamax}(B)
		\end{equation}
		We show \eqref{eq:proof_REGROW_eq3-one-sided} by case distinction. First, consider  $\beta\in[\argwprime,\alphamax]$. In this case, $F_{\argwprime, \beta}$ is convex and decreasing, and thus
		\begin{equation}\label{eq:proof_REGROW_eq4-one-sided}
			\inf_{\beta\geq\argwprime}\inf_{ \mu \geq \mu_1}F_{\argwprime, \beta}^c(\mu) = \inf_{\beta\geq\argwprime}\inf_{\mu \geq \mu_1}F_{\argwprime, \beta}(\mu)= F_{\argwprime,\alphamax}(B),
		\end{equation}
		by monotonicity of $\beta \mapsto F_{\argwprime,\beta}(\mu)$, and $\mu \mapsto F_{\argwprime,\beta}(\mu)$. Similarly, if $\beta\in[0,\argwprime)$, it holds that $x\mapsto F_{\argwprime, \beta}(x)$ is concave and increasing, and thus
		\begin{align}\label{eq:proof_REGROW_eq5-one-sided}
			\begin{split}
				\inf_{\beta\in [0,\argwprime)}\inf_{ \mu \geq \mu_1}F_{\argwprime, \beta}^c(\mu) &= \inf_{\beta\in [0,\argwprime)}\inf_{ \mu \geq \mu_1}G_{\argwprime, \beta}(\mu)\\
				&=\inf_{\beta\in [0,\argwprime)} G_{\argwprime, \beta}(\mu_1)
				=G_{\argwprime, \agw}(\mu_1)= F_{\argwprime,\alphamax}(B),
			\end{split}
		\end{align}
		since for any $\beta\in [0,\argwprime)$, $\mu\mapsto G_{\argwprime,\beta}(\mu)$ is increasing for $\mu\geq \mu_1$, and since $\beta \mapsto G_{\argwprime,\beta}(\mu_1)$, $\beta\in [0,\argwprime)$, has a unique minimum at $\beta=\alpha^\star_{\GW}$. 
		\eqref{eq:proof_REGROW_eq4-one-sided} and \eqref{eq:proof_REGROW_eq5-one-sided} together show \eqref{eq:proof_REGROW_eq3-one-sided}, which concludes the proof. 
	\end{proof}

	\begin{proof}[of Theorem~\ref{thm:master_thm}~(c)]
		Recalling the definition \eqref{eq:deffc} of c-envelope, by Proposition \ref{prop:coin_betting_is_an_optimal_class} $\GROW(\calQhat) = \sup_{\alpha\in I_{\mu_0}}\inf_{\mu\neq\mu_0}f_\alpha^c(\mu)$. For any $\alpha$, by Lemma \ref{lemma:fc} $f^c_\alpha$ is affine since $f_\alpha$ is concave, and so $\inf_{\mu\neq\mu_0}f_\alpha^c(\mu) = \min\{f_\alpha(A), f_\alpha(B)\}\leq 0$. For $\alpha=0$ we have $f_0^c\equiv f_0\equiv 0$, so $\GROW(\calQhat) = 0$, which is achieved by $\Egwhat\equiv 1$. 
		
		As for REGROW, again by Proposition \ref{prop:coin_betting_is_an_optimal_class} we can write
		$$\REGROW(\calQhat)=\sup_{\alpha\in I_{\mu_0}}\inf_{\beta\in I_{\mu_0}}\inf_{\mu\neq\mu_0} F_{\alpha,\beta}^c(\mu)\,.$$
		Recalling that $F_{\alpha,\beta}$ is non-decreasing for $\alpha\geq\beta$, by Lemma \ref{lemma:fc} we have that $\inf_{\mu\neq\mu_0}F^c_{\alpha,\beta}(\mu) = F^c_{\alpha,\beta}(A) = F_{\alpha,\beta}(A)$, whenever $\beta\leq\alpha$ (where we used in the last equality that $\delta_A$ is the only measure in $\scrP$ with mean $A$). Using that $\beta\mapsto F_{\alpha,\beta}(A)$ is non-decreasing, we get that $\inf_{\beta\leq\alpha}\inf_{\mu\neq\mu_0}F^c_{\alpha,\beta}(\mu) = F_{\alpha,\alphamin}(A)$. With an analogous argument for the case $\beta>\alpha$, we obtain that  \begin{equation}\label{eq:supmin}\REGROW(\calQhat) = \sup_{\alpha\in I_{\mu_0}}\min\{F_{\alpha,\alphamin}(A),F_{\alpha,\alphamax}(B)\}\,.\end{equation}
		Writing everything explicitly, we have
		$$F_{\alpha,\alphamax}(B)= \log\left(\frac{\mu_0-A}{B-A}\big(1+\alpha(B-\mu_0)\big)\right)\quad\text{ and }\quad F_{\alpha,\alphamin}(A) = \log\left(\frac{B-\mu_0}{B-A}\big(1+\alpha(A-\mu_0)\big)\right)\,.$$
		Both these expressions are continuous functions of $\alpha$ in $I_{\mu_0}$, $\alpha\mapsto F_{\alpha,\alphamax}(B)$ is strictly increasing and tends to $-\infty$ for $\alpha\to\alphamin$, while $\alpha\mapsto F_{\alpha,\alphamin}(A)$ is strictly decreasing and tends to $-\infty$ for $\alpha\to\alphamax$. It follows that the supremum in \eqref{eq:supmin} is achieved by the unique $\argwhat\in I_{\mu_0}$ satisfying $F_{\argwhat,\alphamax}(B) = F_{\argwhat,\alphamin}(A)$, which is precisely given by \eqref{eq:def_rgw_parameter_for_two_sided}.
		
		Finally, for any $\delta>0$ with $(\mu_0-\delta,\mu_0+\delta)\subseteq (A,B)$, the same reasoning as above yields
		\begin{equation*}
			\GROW(\calQhat_\delta) = \sup_{\alpha\in I_{\mu_0}}\inf_{\mu: |\mu-\mu_0|>\delta}f_\alpha^c(\mu)=\sup_{\alpha\in I_{\mu_0}}\min\{f_\alpha(A), f_\alpha(B)\}=0=\GROW(\calQhat)\,,
		\end{equation*}
		and
		\begin{align*}
			\REGROW(\calQhat_\delta)&=\sup_{\alpha\in I_{\mu_0}}\inf_{\beta\in I_{\mu_0}}\inf_{\mu: |\mu-\mu_0|>\delta} F_{\alpha,\beta}^c(\mu)\\
			&= \sup_{\alpha\in I_{\mu_0}}\min\{F_{\alpha,\alphamin}(A),F_{\alpha,\alphamax}(B)\}=\REGROW(\calQhat) \,,
		\end{align*}
		and thus we obtain the same (RE)GROW e-variables for $\hat{\calQ}_\delta$, independently of $\delta$.
	\end{proof}

	\section{Theorem \ref{thm:master_thm} for $\mu_1<\mu_0$}
	\begin{theorem}\label{thm:master_thm_reversed}
		Let $\mu_0,\mu_1\in (A,B)$ with $\mu_1<\mu_0$.
		\begin{enumerate}[(a)]
			\item 
			For $\calP = \{P\in\scrP\mid\E_P[X]=\mu_0\}$ and $\calQ = \{Q\in\scrP\mid\E_Q[X]=\mu_1\}$, the (RE)GROW e-variables exist and equal $\Egw\equiv E_{\agw}$ and $\Ergw\equiv E_{\argw}$, for 
			\begin{equation}\label{eq:GROW_betting_parameter_reversed}\agw = \frac{\mu_1-\mu_0}{(\mu_0-A)(B-\mu_0)}\,\in (\alphamin,0),\end{equation}
			and $\argw \in (\alphamin,0)$ equals the unique value with $\argw <\alpha_{\GW}^\star$, such that 
			\begin{equation*}F_{\argw,\alphamin}(\mu_1)=G_{\argw,\alpha_{\GW}^\star}(\mu_1)\,.\end{equation*}
			\item  For $\calPprime = \{P\in\scrP\mid\E_P[X]\geq\mu_0\}$ and $\tilde{\calQ} = \{Q\in\scrP\mid\E_Q[X]<\mu_1\}$, the (RE)GROW e-variables exist and equal $E_{\calQprime}^{\textrm{GW}}\equiv E_{\agw}$ and $E_{\calQprime}^{\textrm{RGW}}\equiv E_{\argwprime}$, with $\agw$ given at \eqref{eq:GROW_betting_parameter_reversed} and $\argwprime$ equals the unique value with $\argwprime<\agw$, such that 
			\begin{equation*}
				F_{\argwprime,\alphamin}(A)=G_{\argwprime,\alpha_{\GW}^\star}(\mu_1)\,.
			\end{equation*}
		\end{enumerate}
	\end{theorem}

	\section{Relation with sequential optimality}
	
	Our results concern single-round e-variables and optimise worst-case expected logarithmic growth over a composite alternative specified in advance. A related but distinct line of work considers sequential betting strategies that do not require knowledge of the data-generating alternative. For bounded mean testing, strategies based on empirical plug-ins or universal-portfolio methods adapt their betting parameters to the observations~\citep{waudbysmith2024estimating,orabona2024tight}. More recent results show that strategies satisfying sublinear portfolio regret are universally log-optimal: for every fixed alternative distribution $Q$, their asymptotic logarithmic growth rate matches that of the oracle constant betting parameter tailored to $Q$~\citep{waudbysmith2025universal}. Related instance-optimal guarantees also appear in \citet[Theorem~7.22]{ramdas2025hypothesis} and \citet[Theorem~3]{wang2026ebacktesting}.
	
	There is a natural connection with REGROW, but the two approaches address different problems. Universal-portfolio methods choose the betting parameters adaptively, without specifying an alternative in advance, and aim to perform asymptotically as well as the best fixed betting parameter for the data-generating distribution. REGROW instead selects a single e-variable before observing the data, minimising over a prespecified composite alternative the largest gap from each distribution-specific GRO benchmark. Thus, the former provides an adaptive long-run guarantee, whereas the latter gives a one-round worst-case guarantee.
	
	A natural direction for future work is to extend our analysis to the multi-round setting. Exact REGROW strategies are likely to be difficult to characterise in general, since the optimisation would then range over adaptive betting strategies rather than a single e-variable. This raises the question of whether suitable universal-portfolio or mixture constructions can provide tractable approximations. In particular, the connection between REGROW, minimax redundancy, and Jeffreys-type mixtures identified by \citet{grunwald2024safe} suggests that analogous mixtures over coin-betting parameters may be relevant for repeated bounded mean testing. We leave a formal investigation of this connection to future work.

\end{document}